\documentclass[11pt]{amsart}

\usepackage{amsmath,amssymb}

\usepackage{enumerate}

\usepackage{amsmath}
\usepackage{amssymb}
\usepackage{latexsym}
\usepackage{amsthm}
\usepackage{graphics}

\newcommand{\R}{\mathbb{R}}
\renewcommand{\P}{\mathbb{P}}

\newcommand{\PROB}{{\mathbb{P}}}
\newcommand{\ind}[1]{\mathbb{I}_{\{ #1 \}}}

\theoremstyle{plain}
  \newtheorem{theorem}{Theorem}[section]
  \newtheorem{proposition}[theorem]{Proposition}
  \newtheorem{lemma}[theorem]{Lemma}
    
   \newtheorem{cor}[theorem]{Corollary}
\theoremstyle{definition}
  \newtheorem{remark}[theorem]{Remark}

\edef\@tempa#1#2{\def#1{\mathaccent\string"\noexpand\accentclass@#2
}} \@tempa\ring{017}

\begin{document}
\title{Boundary crossing identities  for  diffusions having the time inversion property}

\author{ L. Alili}
\address{Department of Statistics, The University of Warwick,
Coventry CV4 7AL, United Kingdoms}
\email{l.alili@Warwick.ac.uk}

\author{P. Patie}
\address{Department of Mathematical Statistics and Actuarial Science, University of
Bern, Alpeneggstrasse 22,  CH-3012 Bern, Switzerland}
\email{patie@stat.unibe.ch}



\begin{abstract}
We review and study a one-parameter family of functional transformations,
denoted by $(S^{(\beta)})_{\beta\in \R}$, which,  in the case $\beta<0$,
provides a path realization of bridges associated to the family of diffusion
processes enjoying the time inversion property. This family
includes the Brownian motion,  Bessel processes with
a positive dimension and their conservative $h$-transforms. By
means of these transformations, we derive an explicit and simple
expression which relates the law of the boundary crossing times
for these diffusions over a given function $f$ to those over the
image of $f$ by the mapping $S^{(\beta)}$, for some fixed
$\beta\in \mathbb{R}$. We give some new examples of  boundary
crossing problems for the  Brownian motion and  the family of
Bessel processes. We also
provide, in the Brownian case, an interpretation of the results obtained by the standard
method of images and establish connections between the exact
asymptotics for large time of the densities corresponding to
various curves of each family.
\end{abstract}


\maketitle

\section{Introduction}
\noindent Let $X:=(X_t, t\geq 0)$  be a $E$-valued, where $E= \mathbb{R}$ or
$\mathbb{R}_+=[0,\infty)$,  conservative, in the sense that it has an infinite life-time, $2$-self-similar homogenous diffusion
enjoying the time inversion property in the sense of Shiga and
Watanabe \cite{Shiga-Watanabe-73}.  We denote by
$(\mathbb{P}_x)_{x\in E}$ the family of probability measures of $X$
which act on $\mathcal{C}(\mathbb{R}_+,E)$, the space of
continuous functions from $\mathbb{R}_{+}$ into $E$,  such that
$\P_x(X_0=x)=1$. The standard notation $\mathcal{F}_t$ is used for
the $\sigma$-algebra generated by the process $X$ up to time $t$
and we write simply $\mathcal{F}=\mathcal{F}_{\infty}$. The above assumptions  mean that, for any $x\in E$ and $c>0$,
\begin{equation*} \textrm{the law of }(c^{-1}X_{c^2t},t\geq0;\P_{cx}) \textrm{ is }
\mathbb{P}_x \qquad (\textrm{2-self-similarity})
\end{equation*}
and the process
\begin{equation} \label{eq:ti}
i(X):=(tX_{1/t}, t>0;
\mathbb{P}_x) \qquad(\textrm{time inversion})
\end{equation} is a homogeneous conservative diffusion.    It is well known
that the class of processes considered therein consists of  Brownian
motion and Bessel processes of positive dimensions, see Watanabe
\cite{Watanabe-75}.
Next, let $f\in \mathcal{C}(\mathbb{R}_+,E)$,
such that $f(0) \neq X_0$ and set
$$T^{f} = \inf \left\{s
> 0; \:X_s = f(s) \right\}$$ with the usual convention  that $\inf\{
\emptyset \}=+\infty$. The study of the distribution of the
stopping time $T^{f}$ is known as the boundary crossing or first
passage problem.  Unfortunately, the explicit determination of
these distributions is only attainable for a few specific
functions. In this paper,  we suggest a new method which allows to
relate, in a simple and explicit manner,  the law of $T^{f}$ with
the ones of the family of stopping times $(T^{f^{(\beta)}})_{\beta
\in \mathbb{R}}$ with $f^{(\beta)}:=S^{(\beta)}(f)$ and
\begin{eqnarray} \label{definition-transformation}
S^{(\beta)}: \quad \mathcal{C}(\mathbb{R}_+, E) &
\rightarrow & \mathcal{C}([0,\zeta^{(\beta)}), E) \\ \nonumber f
&\rightarrow& \left(1+\beta .\right)f\left(\frac{.}{1+\beta
.}\right)
\end{eqnarray}
where  \[\zeta^{(\beta)}= \begin{cases} 1/|\beta|, \quad & \textrm{if } \beta <0, \\
+\infty, & \textrm{otherwise}.
 \end{cases}\]
 The results extend to $h$-transforms of the prescribed processes leading to conservative diffusions.

 In the case $\beta<0$, our methodology has a simple description. Indeed, as a consequence of the time inversion property, Pitman and Yor \cite{Pitman-Yor-81} showed that the law  of the process $(S^{(\beta)}(X)_t, 0\leq t< \zeta^{(\beta)})$ corresponds to the law of a bridge associated to $X$.  Thus, one may relate, by means of a deterministic time change, the law of $T^{f}$ with the first crossing time of the curve $f^{(\beta)}$ by the bridge $(S^{(\beta)}(X)_t, 0\leq t< \zeta^{(\beta)})$.  Finally, our identities are then obtained  by using the construction of the law of bridges of  Markov processes as a Doob $h$-transform of their laws, see Fitzsimmons et al.~\cite{Fitzsimmons-Pitman-Yor-93}. We also show that similar devices can be implemented in  the case $\beta>0$.

The motivations for  investigating the boundary crossing problem  are of both practical and
theoretical importance. Indeed, on the one hand, such studies were
originally motivated by their connections to sequential analysis,
see e.g.~Robbins and Siegmund \cite{Robbins-Siegmund-70} and
Smirnov-Kolmogorov test, see e.g.~Lerche \cite{Lerche-86}.  On the
other hand, this problem has  found many applications in several
fields of sciences, such as mathematical physics \cite{Zambrini-Lescot-05},  neurology  \cite{Lanski-Sacerdote-01}, epidemiology
\cite{Martin-Lof-98} and mathematical finance \cite{Bielecki-Jeanblan-Rut-04} and \cite{Patie-th-04}.

Amongst general results, in the Brownian setting, Strassen
\cite{Strassen-67} proved
 that if $f$ is continuously differentiable  then the distribution of $T^{f}$ is absolutely continuous with a
continuous probability density function. Moreover, for elementary
curves, the density is known explicitly for linear, square root and
parabolic functions. For these curves, specific techniques proved to
be efficient and we refer to Table \ref{table} in Section 2  for a
description of these cases. Besides, the distribution is known for
some concave curves solving implicit equations obtained by the
celebrated standard method of images, see Lerche \cite{Lerche-86}.
For some recent investigations, we refer to P\"otzelberger and Wang
\cite{Potzelberger-Wang-01} and Borovkov and Novikov
\cite{Borovkov-Novikov--05} for numerical approximations of the
density, Peskir \cite{Peskir-02-2}, \cite{Peskir-02} for the study
of the small time behavior of the density and Kendall et al.~
\cite{Kendall-Martin-Robert-04b} for statistical applications.

The remaining part of the paper is organized as follows. In the next
Section we recall some recent results regarding diffusions which
enjoy the time inversion property. Section \ref{Mainresults}
is devoted to the statement of our main results and their proofs. Section \ref{Brownian-motion} is concerned
with a detailed study of the Brownian motion case. In particular, we present several new explicit examples of the boundary crossing problem. We also show
how our results translate and agree with the standard method of
images. Finally, in Section \ref{Section-Application-Bessel}, we
treat the example of Bessel processes and characterize the
distribution of hitting times of straight lines. We also mention that, in the
Brownian motion case, the boundary crossing identity
\eqref{switching-identity}, stated below,  has been published
without proofs in the note \cite{Alili-Patie-cras-05}.

\section{Preliminaries}

Let us recall that $E$ is either $\R^+$ or $\mathbb{R}$ and
 $X:=(X_t, t\geq 0)$ is a  $2$-self-similar conservative
homogenous diffusion enjoying the time inversion property.  In
recent papers,  Gallardo and Yor \cite{Gallardo-Yor-05} and Lawi
\cite{Lawi-08} characterized the class of self-similar Markov
processes (which might have c\`adl\`ag paths) satisfying the time
inversion property in terms of their semi-groups when the latter are
assumed to be absolutely continuous and twice differentiable. More
precisely, they showed that if we write $\P_x(X_t\in
dy)=p_t(x,y)dy$ the
transitions densities have the following form
 \begin{equation*}
p_t(x,y) = \frac{c}{\sqrt{t}} \Phi\left(\frac{xy}{t}\right)
\left(\frac{y}{\sqrt{t}}\right)^{2\nu+1} e^{-\frac{x^2+y^2}{2t}}, \quad t>0,\: x,y\in E,
\end{equation*}
where the reals $\nu$ and $c$ are related to the function
$\Phi:E\rightarrow \mathbb{R}_+$ as follows.
\begin{enumerate}
\item If $E=\mathbb{R}$, then $X$ is a Brownian motion, $\Phi(y)=e^{y}$ and
necessarily $\nu=-1/2$ and $c=1/\sqrt{2\pi}$.
\item If
$E=\R^+$ then $X$ is Bessel process of dimension $\delta$ for some $\delta>0$,
$c=1$, $\nu=\frac{\delta}{2}-1$ and $\Phi(y)=y^{-\nu}I_{\nu}(y)$ where $I_{\nu}$ is the modified
Bessel function of the first kind of index $\nu$ which  admits the
following power series representation, see e.g.~\cite[Section
5.2]{Lebedev-72},
\begin{equation*}
I_{\nu}(z)=\left(\frac{z}{2}\right)^{\nu}\sum_{n=0}^{\infty}\frac{(z/2)^{2n}}{n!\Gamma(\nu+n+1)},\:
|z|<\infty,\: |\textrm{ arg}(z)|<\pi.
\end{equation*}
Moreover, it is well known that, in this case,  the state $0$ is
\begin{itemize}
\item[-] an entrance boundary if $\delta\geq2$,
\item[-] a reflecting boundary if $0<\delta<2$.
\end{itemize}
\end{enumerate}
Next, recalling that, for every fixed $y\in E$, the mapping
$x\mapsto \Phi(yx)$ is $\frac{y^2}{2}$-excessive, see e.g.~\cite{Borodin-Salminen-02}, we can define a
new family of probability measures as follows
\begin{equation}\label{h-transform}
\hbox{d}\P_x^{(y)}{}_{|{\mathcal{F}_t}}=\frac{\Phi(y X_t)}{\Phi(y
x)}e^{-\frac{1}{2}{y}^2 t} \hbox{d}\P_x{}_{|{\mathcal{F}_t}}, \quad
t>0.
\end{equation}
Note, from the definition of $\Phi$, that $\P_x^{(0)}=\P_x$. Moreover, for any $y \in E$, $(\mathbb{P}_{x}^{(y)})_{x\in E}$ is precisely the family of probability measures of the process $i(X)$ when $X$ starts at $y$. We also recall that $i(X)$ is conservative and satisfies the time inversion property
\eqref{eq:ti}, see \cite{Gallardo-Yor-05}. It follows that $i(X)$ has an absolutely
continuous semi-group with  densities given, for any $t>0$ and $x,a \in E$, by
\begin{eqnarray} \label{eq:dens_h_transf}
p^{(y)}_t(x,a)&=&\frac{1}{\sqrt{t}}\frac{\Phi(ya)}{\Phi(y
x)}\Phi\left(\frac{ax}{t}\right)\left(\frac{a}{\sqrt{t}}
\right)^{2\nu+1}e^{-\frac{1}{2}\left( ty^2+(x^2+a^2)/t \right)}.
\end{eqnarray}
Note, in particular,  the following:
\begin{enumerate}
\item If $E=\mathbb{R}$ then $i(X)$ is a
Brownian motion with drift $y$.
\item If $E=\R^+$ then $i(X)$ is the so-called Bessel process in the
wide sense  introduced by Watanabe \cite{Watanabe-75}.
\end{enumerate}
Now, from the definition  of the family of mappings $(S^{(\beta)})_{\beta \geq 0}$, see \eqref{definition-transformation}, we observe that for any $f\in \mathcal{C}(\R^+,\R)$ and $\alpha, \beta \in\mathbb{R}$, we have $S^{(0)}(f) = f$ and $
S^{(\alpha)}\circ S^{(\beta)}=S^{(\alpha+\beta)}$ and so
 $(S^{(\beta)})_{\beta \geq 0}$ is a semi-group
on $\mathcal{C}(\R^+,\R)$.  Moreover, the family of mappings $(S^{(\beta)})_{\beta<0}$ naturally appears in the
process of
 construction of bridges associated to $X$. To explain the connection, for any fixed $T>0$
and $x\in E$, let us denote by $(\P^{(y)}_{x, z,T}, z\in E )$  a
regular version of the family of conditional probability
distributions $\left(\P^{(y)}_x(. |X_T = z), z \in E\right)$. That is, for some
$z\in E$, $\P^{(y)}_{x,z,T}$ is the law of the bridge of length $T$ associated to $X$, between $x$ and $z$. This, according to Fitzsimmons and al.~\cite{Fitzsimmons-Pitman-Yor-93}, can be obtained from $\P^{(y)}_{x}$ as the following Doob $h$-transform
 \begin{equation}\label{h-bridges}
 \P^{(y)}_{x,z,T}=\frac{p^{(y)}_{T-s}(X_s, z)}{p^{(y)}_{T}(x,z)} \P^{(y)}_{x}  \quad \textrm{on } \mathcal{F}_s,\: s<T.
 \end{equation}
Fitzsimmons \cite{Fitzsimmons-98} observed  that these laws remain
invariant under changes of probability measures of type
\eqref{h-transform}. Finally, we recall that Pitman and Yor \cite[Theorem
5.8]{Pitman-Yor-81}, see also \cite[Theorem 1]{Gallardo-Yor-05},
showed, by means of the time inversion property that, for any $x$, $z\in E$, the processes
\begin{equation} \label{eq:ibs}
\{X_u, 0\leq u<T; \mathbb{P}_{x,Tz,T}^{(y)}\} \textrm{ and }
\{S^{(-1/T)}(X)_u, 0\leq u<T; \mathbb{P}_x^{(z)} \}
\end{equation}
have the same law. Note that, in the Brownian case,  (\ref{eq:ibs}) can
also be seen from the unique decomposition, as a semi-martingale,
in its own filtration, of the Brownian bridge of length
$T=-\frac{1}{\beta}>0$, between $0$ and $0$. Indeed,  for any
 $0\leq t<T$, we can write
\begin{equation*}
  S^{(-1/T)}(X)_t=(T-t)\int_0^{t} \frac{d\tilde{X}_s}{T- s}, \quad X_t= T\int_0^{\frac{Tt}{T - t}}
\frac{d\tilde{X}_s}{T-  s},\quad t<T,
\end{equation*}
where  $\tilde{X}$ is a Brownian motion on $[0, \zeta^{(-\beta)}]$
with respect to the filtration generated by $
(S^{(-\frac{1}{T})}(X)_t, 0\leq t<T)$. Thus, we have
\begin{equation*}
S^{(-1/T)}(X)_t = \tilde{X}_t - \int_0^t
\frac{S^{(-1/T)}(X)_s}{T-s}ds,\quad t < T.
\end{equation*}
which coincides with the canonical decomposition of the standard
Brownian bridge.

\section{Main results and proofs}\label{Mainresults}

\noindent We keep the notation and setting of the previous section and assume, throughout this section, that  $\beta$ is some fixed real. We
proceed by pointing out that the mapping $S^{(\beta)}$ can be
defined similarly on the space of probability measures. For
instance, in the absolutely continuous case, we associate to
$\mu(dt)=h(t)dt$ the image $S^{(\beta)}(\mu)(dt)= h^{(\beta)}(t) dt$
where we recall that $h^{(\beta)}(t):= S^{(\beta)}(h)(t)$. We are
now ready to state our main result.

\begin{theorem} \label{thm} Let $f\in \mathcal{C}(\mathbb{R}_+,
E)$. Then, for any $x$, $y\in E$ such that $f(0)\neq x$, and
$t<\zeta^{(\beta)}$, we have
\begin{eqnarray*}
\P^{(y)}_x(T^{f^{(\beta)}}\in dt)&=& (1+\beta t)^{\nu-2}
\frac{\Phi(yf^{(\beta)}(t))}{\Phi(yf^{(\beta)}(t)/(1+\beta
t))}e^{-\frac{\beta}{2}\frac{y^2t^2}{1+\beta
t}}e^{-\frac{\beta}{2}{\frac{{f^{(\beta)}}^2(t)}{1+\beta
  t}}+\frac{\beta}{2}x^2}
\\
&\times & S^{(\beta)}\left(\P^{(y)}_x( T^{f} \in
dt)\right).\nonumber
\end{eqnarray*}
\noindent The particular case $y=0$ yields
\begin{equation}\label{switching-identity}
\mathbb{P}_x\left(T^{f^{(\beta)}}\in dt\right)=(1+\beta
t)^{\nu-2}e^{-\frac{\beta}{2}{\frac{{f^{(\beta)}}^2(t)}{1+\beta
  t}+\frac{\beta}{2}x^2}} \times
S^{(\beta)}\left(\mathbb{P}_x\left( T^{f} \in dt\right)\right).
\end{equation}
\end{theorem}
Theorem \ref{thm}   simplifies
when the focus is on straight lines. Indeed, if we consider a
constant function $f\equiv a$ where $a\neq 0$ then,  with
$\beta=b/a$, for some $b\in \R$, we have $f^{(b/a)}(t)=a + bt,\: t<
\zeta^{(b/a)}$.
 Note that if
$b<0$ then the support of $T^{a+b\cdot }$ is $(0,-b/a)$. The
previous result reads as follows.
\begin{cor} \label{thm-straight-lines} For any fixed $a,b,x \in E$ such that $a \neq x$ and
all $t<\zeta^{(b/a)}$, we have
\begin{eqnarray*}
\P^{(y)}_x\left(T^{a+b.}\in dt\right)&=&(1+b t/a)^{\nu-2}
\frac{\Phi(y(a+ b
t))}{\Phi(ya)}\exp{-\frac{b}{2}\left(a+bt+\frac{t^2y^2}{a+bt}-
\frac{x^2}{a}\right)}\\
&\times& S^{(b/a)}\left(\P^{(y)}_x( T^{a} \in dt)\right).\nonumber
\end{eqnarray*}
 In particular, when $y=0$, we obtain
\begin{eqnarray*}
\mathbb{P}_x\left(T^{a+b.}\in dt\right)&=&(1+b t/a)^{\nu-2}
\exp{-\frac{b}{2}(a+bt-
x^2/a)}\\
&\times& S^{(b/a)}\left(\P_x( T^{a} \in dt)\right).
\end{eqnarray*}
\end{cor}
Before proving Theorem \ref{thm} we need to prepare two intermediate results. To this end, let us denote by $H^{(\beta,f)}$
the first time when $S^{(\beta)}(X)$ crosses $f$ i.e.
\begin{equation*}
H^{(\beta, f)} = \inf \left\{0<s < \zeta^{(\beta)};
\:S^{(\beta)}(X)_s =f (s) \right\}.
\end{equation*}
The aim of the next result is to relate the stopping times
$H^{(\beta,f)}$ and $T^{f^{(\beta)}}$.
\begin{lemma} \label{doob} The identities
\begin{equation*}
H^{(-\beta,f)}= \frac{T^{f^{(\beta)}}}{1+\beta T^{f^{(\beta)}}}
\quad \textrm{ and } \quad T^{f^{(\beta)}} =
\frac{H^{(-\beta,f)}}{1-\beta H^{(-\beta,f)}}
\end{equation*}
hold almost surely. In particular, we  have
$\left\{H^{(-\beta,f)}
<\zeta^{(-\beta)}\right\}=\left\{T^{f^{(\beta)}}<\zeta^{(\beta)}
\right\}$.
\end{lemma}
\begin{proof} From the definition of $S^{(\beta)}(X)$ and by using  a deterministic time change, we get
\begin{eqnarray*}
H^{(-\beta, f)} &=& \inf \left\{ 0< s < \zeta^{(-\beta)}; \:
X_{\frac{s}{1-\beta s}}
= \frac{f(s)}{1-\beta s}\right\} \\
&=& \inf \left\{ 0< s < \zeta^{(-\beta)}; \: X_{\frac{s}{1-\beta
s}} =
S^{(\beta)}(f)\left( \frac{s}{1-\beta s}\right)\right\}\\
&=&\frac{T^{f^{(\beta)}}}{1+\beta T^{f^{(\beta)}}}.
\end{eqnarray*}
\noindent The second identity is obtained in a similar way. The last
statement follows then by observing that, for any $\beta \in \R$,
we have  $\zeta^{(\beta)}=\frac{\zeta^{(-\beta)}}{1-\beta
\zeta^{(-\beta)}}$ in the limiting sense.
\end{proof}

\noindent The image by the mapping
$S^{(\beta)}$ of any homogeneous Markov process is clearly a
non-homogeneous Markov process. However, the time inversion property
allows to connect the law of $X$ and that of $S^{(\beta)}(X)$ via a
simple time-space $h$-transform.  To be more precise, we state the following result.
\begin{lemma}\label{abs-diff} Let $x, y \in E$ and  $\beta \in \R$. Under $\P_{x}^{(y)} $, the process
$X^{(\beta)}=S^{(\beta)}(X)$, defined on $[0,\zeta^{(\beta)})$, is
a non-homogeneous strong Markov process. Its  law, denoted by
$\mathbb{Q}_{x}^{y, \beta} $, is absolutely continuous with
respect to $\P^{(y)}_x$ with Radon-Nikodym derivative $M_{x}^{y,
\beta}(t,X_t)$, $t<\zeta^{(\beta)}$, given by
\begin{equation*}
M_x^{y,\beta}(t,X_t)=(1+\beta t)^{-\nu-1}\frac{\Phi(y X_t/(1+\beta
t))}{\Phi(y X_t)}e^{\frac{\beta}{2} \frac{y^2 t^2}{1+\beta
t}+\frac{\beta}{2}\frac{X_t^2}{1+\beta
  t}-\frac{\beta}{2}x^2}.
\end{equation*}
\noindent In particular,  for $y=0$, we have
\begin{equation*}
M_x^{0,\beta}(t,X_t)=(1+\beta t)^{-\nu-1}e^{\frac{\beta}{2}\frac{X_t^2}{1+\beta
  t}-\frac{\beta}{2}x^2}.
\end{equation*}

   \end{lemma}
\begin{proof}  We start with the case  $\beta<0$.  Using successively the identities \eqref{eq:ibs} and \eqref{h-bridges}, we obtain, for any  bounded measurable functional $F$ and $t<\zeta^{(\beta)}$, that
\begin{eqnarray*}
\mathbb{E}_x^{(y)}[F(S^{(\beta)}(X)_s, s\leq t)]&=&\mathbb{E}_x^{(y)}[F(X_s, s\leq t)\mid X_{-1/\beta}=-y/\beta] \\
&=&\mathbb{E}_x^{(y)}\left[\frac{p^{(y)}_{-1/\beta - t} (X_t,-y/\beta)}{p^{(y)}_{-1/\beta} (x,-y/\beta)}F(X_s, s\leq t)\right]\\
&=&\mathbb{E}^{(y)}\left[M_x^{y,\beta}(t,X_t)F(X_s, s\leq t)\right]
\end{eqnarray*}
where the last line follows from \eqref{eq:dens_h_transf}.  Using the fact that, for any
$ c,d \in E$, $\frac{\Phi(cy)}{\Phi(dy)}\rightarrow 1$ as $y\rightarrow
0$, we get the stated result for $y=0$.
 Next, we treat the case $\beta>0$ and assume, first, that $X$ is a Bessel process of
dimension $\delta>0$ and $y=0$. Using the It\^o formula, we see that the
process $(M_x^{0,\beta}(t,X_t), t\geq 0)$ is a
$\mathbb{P}_x$-local martingale. We now show that it is a
true martingale. Recall that the squared Bessel process $X^2$ is
the unique solution to the stochastic differential equation (for
short sde)
\begin{equation}\label{squared-Bessel}
dX_t^2=2 X_tdB_t +\delta dt,\quad X_0^2=x^2,
\end{equation}
where $B$ is a Brownian motion, see \cite{Revuz-Yor-99}.
Furthermore, we need to introduce the processes $Y$ and $\tilde{Y}$ which
are defined, for a fixed $t\geq0$,   by  \[Y_t
=(1+\beta t)\tilde{Y}_t=S^{(\beta)}(X)_t .\] It follows from
(\ref{squared-Bessel}) and by performing a deterministic time
change, that $\tilde{Y}$ solves the sde
\[
d\tilde{Y}_t^2= \frac{\delta
}{(1+\beta t)^2 }dt+\frac{2\tilde{Y}_t}{1+\beta t}d\gamma_t
\]
where $(\gamma_t, t\geq0)$ is a
Brownian motion with respect to the filtration generated by the process $\tilde{Y}$. Similarly, $Y^2$
satisfies
\begin{eqnarray}\label{intermediate-equation}
dY_t^2&=&d\left((1+\beta t)^2\tilde{Y}_t^2 \right)   \\\nonumber
 &=&
2Y_t \left(d\gamma_t +\frac{\beta}{1+\beta t} Y_tdt\right) +
\delta dt.
\end{eqnarray}
Now, from \eqref{squared-Bessel} and  by Girsanov theorem, the law
of the unique solution to (\ref{intermediate-equation}) is
obtained from the law of $X^2$ by a change of probability measure
using the local martingale $(M_{x}^{0, \beta}(t,X_t),
t\geq0)$. The proof of the claim, in the case $y=0$, is completed by invoking the conservativeness property of $Y^2$,
which implies that $(M_{x}^{0, \beta}(t,X_t), t\geq0)$ is
a  martingale. Finally, to recover the case $y\neq 0$, we use
(\ref{h-transform}) which gives
\begin{eqnarray*}
\hbox{d}\mathbb{Q}_x^{y,\beta}|_{\mathcal{F}_t}&=&\frac{\Phi(y X_t
/(1+\beta t))}{\Phi(y x)}e^{-\frac{1}{2}
\frac{ty^2}{1+\beta t}}\hbox{d}\mathbb{Q}_x^{0,\beta}|_{\mathcal{F}_t}\\
&=& \frac{\Phi(y X_t/(1+\beta t))}{\Phi(y x)}e^{-\frac{1}{2}
\frac{ty^2}{1+\beta t}} M_{x}^{0,
\beta}(t,X_t)\hbox{d}\mathbb{P}_x|_{\mathcal{F}_t}\\
&=& \frac{\Phi(y X_t/(1+\beta t))}{\Phi(y X_t)}e^{\frac{\beta}{2}
\frac{y^2 t^2}{1+\beta t}} M_{x}^{0,
\beta}(t,X_t)\hbox{d}\mathbb{P}_x^{(y)}|_{\mathcal{F}_t}\\
 &=&M_{x}^{y,\beta}(t,X_t)\hbox{d}\mathbb{P}_x^{(y)}|_{\mathcal{F}_t}.
 \end{eqnarray*}
This completes the proof in the Bessel case. The same arguments work for the Brownian  case and this completes the proof.
\end{proof}

\noindent We shall now return back to the proof of Theorem \ref{thm}. Lemmas
\ref{doob} and \ref{abs-diff}, when combined with the optional
stopping theorem, allow us to write, for any $\lambda>0$,
\begin{eqnarray*}
\mathbb{E}^{(y)}_x\left[e^{-\lambda T ^{
f^{(\beta)}}}\ind{T^{f^{(\beta)}}<\zeta^{(\beta)}} \right]
&=&\mathbb{E}^{(y)}_x\left[
e^{-\lambda\frac{H ^{(-\beta,f)}}{1-\beta H ^{(-\beta,f)}}}\ind{H ^{(-\beta,f)}<\zeta^{(-\beta)}} \right]\\
&=&\mathbb{E}^{(y)}_x\left[e^{-\lambda\frac{T^{f}}{1-\beta
T^{f}}}M_{x}^{y,-\beta}(T^{f},X_{
T^{f}}) \ind{T^{f}<\zeta^{(-\beta)}} \right]\\
&=&\int_0^{\zeta^{(-\beta)}} e^{-\lambda\frac{t}{1-\beta
t}}M_{x}^{y,-\beta}(t,f(t))\P^{(y)}_x\left(T^{f}\in dt \right)\\
&=&\int_0^{\zeta^{(\beta)}} \frac{e^{-\lambda u}}{(1+\beta u)^3}
M_{x}^{y,-\beta}\left(u/(1+\beta u), f(u/(1+\beta u)\right)\\
&\times& S^{(\beta)}\left(\P^{(y)}_x( T^{f} \in du)\right)
\end{eqnarray*}
where we have performed the change of variables $t/(1-\beta t)=u$. Simplifying the above expression and by
the injectivity of the Laplace transform, we get our first
assertion. The second statement follows by letting $y\rightarrow 0$.


\section{Brownian motion}\label{Brownian-motion}
\noindent In this paragraph, we take $\mathbb{E}=\mathbb{R}$. Thus, $X$ is a Brownian motion. Thanks to the homogeneity property, in this case, it is
clearly enough to consider the case $X_0=0$ and we simply write $\P$
for $\P_0$.
To simplify the discussion, except in Subsection \ref{asymptotics-Section}, we assume that $f\in
\mathcal{C}^{1}(\mathbb{R}_+,E)$, with $f(0)\neq 0$, which implies
that the studied distributions are absolutely continuous with
respect to the Lebesgue measure with continuous densities, see \cite{Strassen-67}. We write
then $\mathbb{P}\left(T^{f}\in dt\right)=p^{f}(t)\: dt$ and read
from Theorem \ref{thm}   the identity
\begin{equation}\label{identity-brownian}
p^{f^{(\beta)}}(t)=\frac{1}{(1+\beta t)^{3/2}}e^{ -\frac{1}{2}
\frac{\beta}{ 1+\beta
t}f^{(\beta)}(t)^2}p^{f}\left(\frac{t}{1+\beta t}\right),\: t<\zeta^{(\beta)}.
\end{equation}
Observe, when $\beta>0$, the asymptotic behavior
\begin{equation}\label{asymptotic-0}
p^{f^{(\beta)}}(t)\sim (\beta t)^{-3/2}e^{-\frac{1}{2}\frac{\beta}{1+\beta t}{f^{(\beta)}}^2(t)}p^{f}(1/\beta) \quad \textrm{ as } t\rightarrow \infty,
\end{equation}
where $h(t)\sim g(t)$ as $t\rightarrow \zeta$ means that $h(t)/g(t)\rightarrow 1$ as
$t\rightarrow \zeta$, for some $\zeta\in [0,\infty]$. It seems natural to examine the images of some curves by the
mapping $S^{(\beta)}$. We gathered in Table \ref{table} the images
of the most studied curves.
\begin{table}[h]
\begin{center}
\begin{tabular} {|c|c|c|c|}  \hline\label{tab:cur}
 & $f$ & $f^{(\beta)}$& $\textrm{References}$ \\
\hline \hline
(1)& $ a + bt$ &$ a + (b+a\beta)t$ & \cite{Bachelier-41} \\
(2) & $\sqrt{1+2 b t}$ & $\sqrt{(1+\beta t)(1+(\beta + 2 b)t)}$&
\cite{Breiman-67}, \cite{Novikov-71}
\\
 (3) &  $ a+bt^2, \:ab>0 $ & $a(1+\beta t)+ bt^2/{(1+\beta
t)}$& \cite{Groeneboom-89}, \cite{Salminen-88}\\  
(4) & $ \frac{a}{2} -\frac{t}{a}\ln\left(\frac{b + \sqrt{b^2 + 4
b_1 e^{-\frac{a^2}{t}}}}{2} \right) $& $\frac{a ( 1+ \beta t)}{2}
-\frac{t}{a}\ln\left(\frac{b +
\sqrt{b^2 + 4 \hat{b}_1 e^{-\frac{a^2}{t}}}}{2}\right)$ \vspace{0.1cm}& \cite{Daniels-69} \\
\hline
\end{tabular}
\end{center}
\caption{Image by $S^{(\beta)}$ of $f$  and the corresponding references where the distribution of
$T^{f}$ is studied with  $a,b \in \R$,
$b_1>-b^2/4$ and $\hat{b}=b_1e^{-a^2\beta}$. } \label{table}
\end{table}
The fact that $S^{(\beta)}$ preserves straight lines is well known, see for instance \cite{Pitman-Yor-81}. More generally, we observe that if $g_{\beta, \alpha}(.)=(1+\beta .)^{\alpha}$, for some reals $\alpha$ and $\beta$, then we have $g^{(-\beta)}_{\beta, \alpha}=g_{-\beta, 1-\alpha}$. In particular, if we  take $f\equiv a$ and $\beta=b/a$, for some reals $a$ and $b$, then
$f^{(b/a)}(t)=a+bt$, $t<\zeta^{(b/a)}$. An immediate application of
Corollary \ref{thm-straight-lines},  combined with
\begin{equation*}
\mathbb{P}(T^{ a}\in dt)=\frac{|a|}{\sqrt{2\pi
t^3}}e^{-\frac{a^2}{2t}}dt,\: t>0,
\end{equation*}
 yields then the following  well known Bachelier-L\'evy formula
\begin{equation*}
\P\left(T^{a+b\cdot}\in dt \right)=\frac{|a|}{\sqrt{2 \pi
t^3}}e^{-b a -\frac{b^2}{2}t-\frac{a^2}{2t}}\: dt.
\end{equation*}

\subsection{Some new examples}\label{examples}

We now consider the boundary crossing problem associated to  two families of
curves consisting of the square root of  second order polynomials and
the reciprocal of affine functions.

\subsubsection{}
First, we consider the distribution of the stopping time
\begin{equation*}
T^{(\lambda_1, \lambda_2)}_a=\inf\left\{ s> 0;\:
X_s=a\sqrt{\left(1+\lambda_1 s\right)\left(1+\lambda_2 s
\right)}\right\}
\end{equation*}
where $a$ and $\lambda_1<\lambda_2$ are some fixed reals. We do not
consider the case $\lambda_1=\lambda_2$ which can be studied in a
elementary way, with the extra cost of making use of the strong
Markov property when $\lambda_1<0$. First, consider the case $\lambda_2=0$ and, to simplify the notation,
set $\lambda_1=\lambda$ and $T_a^{(\lambda,0)}=T_a^{(\lambda)}$.
This is the setting of the classical example studied by Breiman in
\cite{Breiman-67}. $T_a^{(\lambda)}$ is related to the hitting time of a constant
level by an Ornstein-Uhlenbeck process and  we refer to Alili et
al.~\cite{Alili-Patie-Pedersen-05} for a recent survey on this
topic. That is, with \begin{equation*} U_t=e^{-\lambda
t/2}\int_0^te^{\lambda s/2}dX_s, \quad  t\geq 0, \end{equation*} and
$H_a=\inf\{ s> 0; \: U_s=a\}$, we have
\begin{equation} \label{eq:iht}
T^{(\lambda)}_a \stackrel{(d)}{=}\lambda^{-1}\left( e^{\lambda
H_a}-1 \right)
\end{equation}
where $\stackrel{(d)}{=}$ stands for the identity in distribution.
By symmetry, it is enough to consider the case where $a$ is
positive. We proceed by  recalling that the distribution of $H_a$ is given,
see for instance \cite{Alili-Patie-Pedersen-05},  by
\begin{eqnarray*}
 {\mathbb P}\left(H_a\in dt \right)   = -\frac{1}{2} \lambda e^{-\lambda a^2/4}
 \sum_{n=1}^{\infty} \frac{D_{\nu_{n,- a\sqrt{\lambda}}}(0)}
 {D^{(\nu)}_{ \nu_{n,- a\sqrt{\lambda}}}(- a \sqrt{\lambda})} \: e^{-\lambda \nu_{n, -a\sqrt{\lambda}}
 t/2},\quad t>0,
\end{eqnarray*}
where we used the notation $D^{(\nu)}_{\nu_{n,b}}(b) =
\frac{\partial D_{\nu }(b)}{\partial\nu}|_{\nu= \nu_{n,b}}$ and
$(\nu_{j,b})_{j\geq 0}$ stands for the ordered sequence of the
positive zeros of the parabolic function $\nu \rightarrow
D_{\nu}(b)$. By means of the identity \eqref{eq:iht}, we get that
\begin{equation*}
{\mathbb P}\left(T^{(\lambda)}_a\in dt \right)=\frac{1}{1+\lambda
t}{\mathbb P}\left(H_a\in d\cdot \right){\big | }_{\cdot=
\frac{1}{\lambda}\log\left(1+\lambda t \right)}dt,\quad t>0.
\end{equation*}
Next, we assume that $\lambda_1<\lambda_2$. Thus, the support of
$T^{(\lambda_1, \lambda_2)}_a$ is $[0,\zeta^{(\lambda_1)})$   if
$\lambda_1$ is positive and is the positive real line otherwise. We
have
\begin{equation*}
S^{(\lambda_1)}\left(\sqrt{1+(\lambda_2-\lambda_1)\cdot}
\right)=\sqrt{\left(1+\lambda_2 \cdot\right)\left(1+\lambda_1
\cdot \right)}.
\end{equation*}
We are now ready to use Theorem \ref{thm} and write
\begin{equation*}
\mathbb{P}\left(T^{(\lambda_1, \lambda_2)}_a\in
dt\right)=\frac{1}{(1+\lambda_1 t)^{5/2}}e^{ -\frac{1}{2} \lambda_1(
1+\lambda_2 t)} S^{(\lambda_1)}\left(\mathbb{P}\left(
T^{\left(\lambda_2-\lambda_1\right)}_a \in dt\right)\right),\quad
t<\zeta^{(\lambda_1)}.
\end{equation*}

\subsubsection{} We are now interested in computing
the distribution of the stopping time  $T^{h^{(\beta)}}$ defined by
\begin{equation*}
T^{h^{(\beta)}}=\inf\left\{ 0<s<\zeta^{(\beta)} ;\:
X_s=\frac{1}{1+\beta s}\right\}
\end{equation*}
where $\beta$ is some real. To this end, we recall that
Groeneboom \cite{Groeneboom-89} has computed the density of
$T^{\tilde{h}}$ with $\tilde{h}(t)=1+\beta^2t^2$ as follows
\begin{equation*}
\P(T^{\tilde{h}} \in dt) = 2(\beta^2 c)^{2} e^{-\frac{2}{3}\beta^4
t^3}
\sum_{k=0}^{\infty}\frac{Ai\left(z_k+2c\beta^2\right)}{Ai'(z_k)}e^{-z_k
t}\:dt, \: t>0,
\end{equation*}
where $(z_k)_{k \geq 0}$ is the decreasing sequence of negative
zeros of the Airy function, see e.g.~\cite{Lebedev-72}, and we have
set $c=(2\beta^2)^{-\frac{1}{3}}$. Next,  by means of  the
Cameron-Martin formula, we obtain, with $h(t)=(1-\beta t)^2=\tilde{h}(t)-2\beta t$,
\begin{equation*}
\P(T^{{h}} \in dt) = 2(\beta^2 c)^{2} e^{2\beta-2\beta^2 t
\left(1+\frac{2}{3}\beta^2 t^2-\beta t\right)}
\sum_{k=0}^{\infty}\frac{Ai\left(z_k+2c\beta^2\right)}{Ai'(z_k)}e^{-z_k
t}\:dt,\: t>0.
\end{equation*}
Finally, since $h^{(\beta)}= S^{(\beta)}(h)$, we obtain, by applying Theorem \ref{thm}, that
\begin{equation*}
\P(T^{{h^{(\beta)}}} \in dt) = \frac{2(\beta^2 c)^{2}
e^{2\beta}}{(1+\beta t)^{3/2}}e^{-2\beta^2 t
\left(1+\frac{1}{4\beta}+\beta t +\frac{2}{3}\beta^2 t^2\right)}
\sum_{k=0}^{\infty}\frac{Ai\left(z_k+2c\beta^2\right)}{Ai'(z_k)}e^{-z_k
\frac{t}{1+\beta t}}\:dt, \quad t<\zeta^{(\beta)}.
\end{equation*}
We complete the example by stating the following asymptotic result
\begin{equation*}
\P(T^{{h^{(\beta)}}} \in dt) \sim \left(2(\beta^2 c)^{2}
e^{2\beta}\sum_{k=0}^{\infty}\frac{Ai\left(z_k+2c\beta^2\right)}{Ai'(z_k)}e^{-
\frac{z_k}{\beta}}\right) e^{-\frac{4}{3}\beta^4 t^3} \:dt\:
\textrm{ as } t\rightarrow \infty
\end{equation*}
which holds provided that $\beta>0$.

\subsection{Interpretation of the mapping $S^{\beta}$ via the method of images}\label{images}
 We aim to describe the impact of our methodology to the so-called standard method of images.  To this end, let us assume that $X_0= 0$ and set
\[h(t,x)dx=\mathbb{P}\left( T^{f}>t, X_t \in
dx\right).\] We need to impose some conditions on the the curve $f$ in order to be able to apply the standard method of images. That is, we assume that
 \begin{itemize}
 \item[-] $f$ is infinitely often continuously  differentiable,
\item[-] $f(t)/t$ is monotone decreasing,
\item[-] $f$ is concave.
\end{itemize}
Note that these properties are also satisfied  by $f^{(\beta)}$ when $\beta >0$. The function $h$ is characterized as being the unique solution to the
heat equation, see Lerche \cite[Chap.~I.1]{Lerche-86},
\begin{equation*}
\frac{\partial h}{\partial t}=\frac{1}{2}\frac{\partial^2
h}{\partial x^2} \quad \textrm{on } \{(t,x)\in
\mathbb{R}^+\times \mathbb{R};\:x\leq f(t) \},
\end{equation*}
 with the boundary conditions
\begin{equation*}
h\left(t, f(t)\right)=0,\quad h(0, .)=\delta_0(.)\quad
\hbox{on}\quad ]-\infty, f(0^+)]
\end{equation*}
where $\delta_0$ stands for the Dirac function at $0$ and
$f(0^+)=\lim_{t\searrow 0}f(t)$. The standard method of images
assumes that $h$ is known, whilst $f$ is unknown, and is given by
\begin{equation}\label{space-time} h(t,x)=\frac{1}{\sqrt{2\pi t}}\left(
e^{-\frac{x^2}{2t}}-\int_{0}^{\infty}e^{-\frac{(x-s)^2}{2t}}
F(ds)\right)
\end{equation}
where $F(ds)$ is some positive $\sigma$-finite measure satisfying
$\int_{0}^{\infty}e^{-\frac{\epsilon s^2}{2}} F(ds)<\infty$, for all
$\epsilon>0$. In \cite{Lerche-86}, it is shown that if $f$ is the
unique root of the equation $h(t,x)=0$ in the
unknown $x$ for a fixed $t\geq 0$, then we have
\begin{equation}\label{hitting-density-brut}
\mathbb{P}\left(T^{f}\in dt \right)=\frac{dt}{\sqrt{2\pi
t^3}}\int_{0}^{\infty}se^{-\frac{(s-f(t))^2}{2t}}F(ds),\quad t>0.
\end{equation}
Note that the implicit equation $h(t, x)=0$ may be written as
\begin{equation} \label{impl}
\int_{0}^{\infty}e^{-\frac{s^2}{2t}+s\frac{x}{t}}F(ds)=1.
\end{equation}
 With $F(ds)$ replaced by $F(ds)e^{-\mu s}$, $\mu>0$, the
unique solution to \eqref{impl} is $f(t)+\mu t$ for a fixed positive
real $t$ and this is easily checked by the Cameron-Martin formula.
In the same spirit, the identity (\ref{identity-brownian}) has the following interpretation.

\begin{proposition}\label{images-interpretation} If $\beta>0$ then we have the following representation
\begin{equation*}\mathbb{P}\left(T^{f^{(\beta)}}\in dt \right)=\frac{dt}{\sqrt{2\pi
t^3}}\int_{0}^{\infty}se^{-\frac{(s-f^{(\beta)}(t))^2}{2t}}e^{-\beta \frac{s^2}{2}}F(ds), \quad t>0.
\end{equation*}
Furthermore, the $\sigma$-finite measure corresponding to the curve $f^{(\beta)}$ is $F^{(\beta)}(ds):=e^{-\beta s^2/2}F(ds)$ and
\begin{equation*}
\mathbb{P}\left( T^{f^{(\beta)}}>t, X_t\in
dx\right)=\frac{dx}{\sqrt{2\pi t}}\left(
e^{-\frac{x^2}{2t}}-\int_{0}^{\infty}e^{-\frac{(x-s)^2}{2t}}
e^{-\beta^2 \frac{s}{2}}F(ds)\right).
\end{equation*}
 \end{proposition}

\begin{proof} Noting, from  (\ref{hitting-density-brut}), that
\begin{eqnarray*}
\mathbb{P}\left(T^{f^{(\beta)}}\in dt \right)&=& \frac{dt}{\sqrt{2\pi
t^3}}\int_{0}^{\infty}se^{-\frac{(s-f^{(\beta)}(t))^2}{2t}}F^{(\beta)}(ds)
\end{eqnarray*}
and by using  (\ref{identity-brownian})
we obtain that
\begin{eqnarray*}
\frac{1}{\sqrt{2\pi
t^3}}\int_{0}^{\infty}se^{-\frac{(s-f^{(\beta)}(t))^2}{2t}}F^{(\beta)}(ds)
&=&\frac{1}{(1+\beta t)^{3/2}}e^{ -\frac{1}{2}
\frac{\beta}{ 1+\beta
t}f^{(\beta)}(t)^2}p^{f}\left(\frac{t}{1+\beta t}\right)\\
&=&\frac{1}{\sqrt{2\pi t^3}}e^{ -\frac{1}{2}
\frac{\beta}{ 1+\beta
t}f^{(\beta)}(t)^2}\int_{0}^{\infty}se^{-\frac{(s-f(t/(1+\beta t)))^2}{2t/(1+\beta t)}}F(ds)\\
&=&\frac{1}{\sqrt{2\pi
t^3}}\int_{0}^{\infty}se^{-\frac{(s-f^{(\beta)}(t))^2}{2t}}e^{-\beta \frac{s^2}{2}}F(ds).
\end{eqnarray*}
This proves the first assertion and  suggests, but doest not prove, that the $\sigma$-finite measure corresponding to the curve $f^{(\beta)}$ is $F^{(\beta)}(ds)$. Next, with
\begin{equation*}
h^{(\beta)}(t,x)=\frac{1}{\sqrt{2\pi t}}\left(e^{-\frac{x^2}{2t}}-\int_{0}^{\infty}e^{-\frac{(x-s)^2}{2t}}
e^{-\beta^2 s/2}F(ds)\right),
\end{equation*}
we aim to solve $h^{(\beta)}(t,x)=0$ for each fixed $t>0$. Now, setting $x=f(t)$ and replacing  $x$ and $t$, respectively, by $f(t)$ and $t/(1+\beta t)$ in (\ref{impl}), we find that  $\int_{0}^{\infty}e^{-\frac{s^2}{2t}+s\frac{f^{(\beta)}(t)}{t}}e^{-\beta s^2/2}F(ds)=1$. It follows that $f^{(\beta)}(t)$ solves the implicit equation $h^{(\beta)}(t,x)=0$, $t>0$, which finishes the proof.
\end{proof}

\subsection{Large asymptotics for the density of $T^{f^{(\beta)}}$} \label{asymptotics-Section} Following Anderson and Pitt \cite{Anderson-Pitt-97},
we consider the asymptotic  of the density of the
  distribution of $T^{f^{(\beta)}}$ as time gets close to $\zeta^{(\beta)}$. To start with, we exceptionally assume in this subsection that $f\in \mathcal{C}^1((0,\infty), \mathbb{R}_+)$.  So the distribution of $T^{f}$ has a continuous density with respect to the Lebesgue measure on $(0,\infty)$.  Assume that
 the distribution of $T^{f}$ is
defective. That is $r=\P(T^{f}<\infty)<1$ and, in this case, we say that $f$ is transient. By the classical Kolmogorov-Erd\"os-Petrovski theorem, see
\cite{Erdos-42}, we know that if   $t^{-1/2}f(t)$ is increasing
for sufficiently large $t$, then $f$ is transient  if and only if the integral test
\begin{equation}\label{integral-test}
\int_1^{\infty}t^{-3/2}f(t)e^{-f^2(t)/2t}\: dt<\infty
\end{equation}
holds.  Clearly, if $\beta>0$ and $0<\beta f(1/\beta)<\infty$ then  $f^{(\beta)}$ does not satisfy (\ref{integral-test}) but the general formula (\ref{asymptotic-0}) holds. Now, we consider the more interesting case $\beta<0$ and examine the asymptotic of the density of the distribution of $T^{f^{(\beta)}}$ as $t\rightarrow -1/\beta$. We need to work under the following three conditions borrowed from
\cite{Anderson-Pitt-97}:
\begin{itemize}
\item[-] $f$ is increasing, concave, twice differentiable on
$(0,\infty)$ and of regular variations at $\infty$ with index
$\alpha\in [1/2,1)$,
 \item[-]
 $f(t)/\sqrt{t}$ is monotonic increasing at $\infty$, and $f(t)/t$
 is convex and decreases to $0$ for sufficiently large $t$,
 \item[-] There exist positive constants $c<1$ and $c'$ such
 that  we have the inequalities
 $tf'(t)<cf(t)$ and $|t^2 f''(t)|\leq c' f(t)$ for a sufficiently large enough $t$.
\end{itemize}
  The behavior at $\infty$ imposed in the
above conditions is granted in the examples where $f$ behaves like
$$ f(t)=Ct^a(\log t)^b(\log\log t)^c
(\log\log\log t)^d$$
 with $1/2\leq
a<1$, for large $t$. It is clear that these conditions are not always preserved by the family $S^{(\beta)}$.  Equipped
with this, we are now ready to state the following result.
\begin{proposition}\label{asymptotics}  Assume that $f$ is transient and satisfies the above conditions. Then,
 for any $\beta<0$, we have
\begin{equation*}
p^{f^{(\beta)}}(t) \sim \frac{1}{\sqrt{2\pi |\beta|^{3}}}\left(1-r \right)\tilde{f}(\beta, t) \quad \textrm{as } t\rightarrow \zeta^{(\beta)}
\end{equation*}
where
\begin{equation*}
\tilde{f}(\beta, t)=f\left(\frac{t}{1+\beta
t}\right)- \frac{t}{1+\beta
t}f'\left(\frac{t}{1+\beta t}\right).
\end{equation*}

\end{proposition}
\begin{proof} We read from Theorem 1 that
\begin{equation} \label{eq:as_d}
p^{f}(t) \sim \left(1-r
\right)\frac{f(t)-tf'(t)}{\sqrt{2\pi}t^{3/2}}e^{-f^2(t)/2t} \quad \textrm{as } t\rightarrow \infty.
\end{equation}
 Moreover, as $t\rightarrow \infty$,  $f(t)/t \downarrow 0$ and  hence
\begin{equation*}
e^{-\frac{\beta}{2(1+\beta t)}{f^{(\beta)}}^2(t)-\frac{1+\beta t}{2t}f^2(t/(1+\beta t))}=e^{-{f^{(\beta)}}^2(t)/2t}\sim 1.
\end{equation*}
The proof is then completed by combining \eqref{eq:as_d} with (\ref{identity-brownian}).
\end{proof}

\begin{remark}
 In the defective case, the
distribution of the last crossing time $\tilde{T}^{f}=\sup\{s>0;
X_s= f(s) \}$ is shown to have an atom at 0 and its asymptotic as
$t\rightarrow 0$ is determined in \cite{Strassen-67}. Similar
questions can be treated for $\tilde{T}^{f^{(\beta)}}=\sup\{s>0;
X_s= f^{(\beta)}(s) \}$ using this method.
\end{remark}

\section{Bessel processes and straight lines}\label{Section-Application-Bessel}
\noindent We investigate here the case when $X$ is a Bessel process
of dimension $\delta>0$ and we refer to Revuz and Yor \cite[Chap.
XI]{Revuz-Yor-99} for a concise treatment of these processes. For
$\delta \geq 2$, and $x>0$, the Bessel process of dimension $\delta$
is the unique solution to
\begin{equation*}
 X_t = B_t +x+\frac{\delta -1}{2}\int_0^t
\frac{ds}{ X_s},\quad t> 0,
\end{equation*}
where $B$ is a standard Brownian motion. For $0<\delta<2$, $X$ is
defined as the square root of the unique non-negative solution of
\eqref{squared-Bessel}. We recall that $0$ is
an entrance boundary when $\delta\geq2$ and a reflecting boundary
otherwise. We denote by $\mathbb{P}_{x}^{\nu}$ (resp.~$\mathbb{E}_x^{\nu}$) the law  (resp.~the expectation operator) of
a Bessel process of index $\nu=\delta/2-1$, starting at $x>0$. The semi-group of  $X^2$ is characterized  by
\begin{eqnarray} \label{eq:sg}
\mathbb{E}_x^{\nu}\left[ e^{- \lambda X^2_t} \right] &=& (1+2\lambda
t)^{-\delta/2} e^{- \frac{ \lambda x^2}{1+2\lambda t}}, \quad
\lambda,t\geq 0.
\end{eqnarray}
The densities  of the semi-group of $X$, with respect to the
Lebesgue measure, are
\begin{eqnarray*}
p^{\nu}_t(x,y) &=& \frac{y}{t}\left(\frac{y}{x}\right)^{\nu}
e^{-\frac{x^2+y^2}{2t}} I_{\nu}\left( \frac{xy}{t}\right), \quad t,
x,y >0.
\end{eqnarray*}
 For a given $f\in \mathcal{C}(\mathbb{R}_+,E)$ let us
keep the notation used in the introduction and Section
\ref{Mainresults}.  Observe, that if $y=0$ then
 Theorem \ref{thm} reads
\begin{equation*}
\mathbb{P}_x^{\nu}\left(T^{f^{(\beta)}}\in dt\right)={(1+\beta
t)^{\nu -2}}e^{ -\frac{\beta}{2} \frac{f^{(\beta)}(t)^2}{1+\beta
t}+\frac{\beta}{2}x^2}S^{(\beta)}\left(\mathbb{P}_x^{\nu}( T^{f}
\in dt)\right)
\end{equation*}
for all $t<\zeta^{(\beta)}$. We  end up our discussion by
computing the distribution of the hitting time by $X$ of a
straight line $a +b\cdot$ where $a>0$ and $b$ is a real.  We keep
the notation used in Corollary \ref{thm-straight-lines} and the
reader is reminded about observations preceding its statement.
To the best of our knowledge, the problem of the determination of the
distribution of $T^{a+b.}$, which was raised in
\cite{Pitman-Yor-81} for the case $a\neq 0$, remained open.  Recall that the law of $T^{a}$, which
corresponds to $b=0$, is characterized by
\begin{eqnarray}\label{eq:laplace-bessel}
\mathbb{E}_x^{\nu}\left[ e^{- \frac{\lambda^2}{2} T^{a}} \right] =
 \left\{
\begin{array}{lll}
 &  \frac{x^{-\nu}I_{\nu}(x\lambda)}{a^{-\nu}I_{\nu}(a\lambda)}, & x \leq a, \\
 & \\
 & \frac{x^{-\nu}K_{\nu}(x\lambda)}{a^{-\nu}K_{\nu}(a\lambda)}, &   x \geq a, \\
\end{array}
\right.
\end{eqnarray}
for any $\lambda \geq 0$,
where $K_{\nu}$ is the  modified Bessel functions of the second kind
of index $\nu$, see for instance Borodin and Salminen
\cite{Borodin-Salminen-02}. Observe that when $x> a$ the
distribution of $T^a$ is defective and
$\mathbb{P}^{\nu}_x(T^{a}<\infty)=(a/x)^{2\nu}$. It is also well known that if $x<a$ then
we have
\begin{eqnarray}\label{eq:density-bessel}
\mathbb{P}_x^{\nu}(T^{a} \in dt) &=&
 \sum_{k=1}^{\infty}
\frac{x^{-\nu}j_{\nu,k}J_{\nu}(j_{\nu,k}\frac{x}{a})}{a^{2-\nu}J_{\nu+1}(j_{\nu,k})}
e^{-j^2_{\nu,k} t/2a^2} \: dt,\quad t>0,
\end{eqnarray}
where $(j_{\nu,k})_{k\geq1}$ is the ordered sequence of positive
zeros of the Bessel function of the first kind  $J_{\nu}(.)$,  see
\cite{Borodin-Salminen-02}. Furthermore, if $a=0$ and $b>0$ then  $T^{b\cdot}$
and 1/$\sup\{s>0; X_s=b\}$ have the  distribution which was
determined in \cite{Pitman-Yor-81} by making use of the time
inversion property. Now, we are ready to state the following result.
\begin{theorem}  For $0\leq x<a$ and $b\in \mathbb{R}$, we
have for any $t<\zeta^{(b/a)}$
\begin{eqnarray*}
\mathbb{P}_x^{\nu}(T^{a+b\cdot} \in dt) &=&
\frac{e^{\frac{b}{2a}(a^2-x^2)+\frac{b^2}{2}t}}{(1+\frac{b}{a}
t)^{\nu+2}} \sum_{k=1}^{\infty}
\frac{x^{-\nu}j_{\nu,k}J_{\nu}(j_{\nu,k}\frac{x}{a})}{a^{2-\nu}J_{\nu+1}(j_{\nu,k})}
e^{-j^2_{\nu,k}\frac{t}{2a(a+bt)}} \: dt.
\end{eqnarray*}
For any $x\geq 0$ and $a, b >0$,  we have
\begin{equation*}
\mathbb{E}_x^{\nu}\left[ e^{- \frac{\lambda^2}{2} T^{a-b\cdot}}
\right] =
 \left\{
\begin{array}{lll}
 &  e^{\frac{b}{2a}(a^2-x^2)}  \frac{x^{-\nu}}{a^{-\nu}} \int_0^{\infty}  \frac{I_{\nu}(\sqrt{2}xu)}{I_{\nu}(\sqrt{2}au)}
p^{\nu}_{b/2a}(\lambda_{b},u)\: du , & x \leq a, \\
 & e^{\frac{b}{2a}(a^2-x^2)}  \frac{x^{-\nu}}{a^{-\nu}} \int_0^{\infty}  \frac{K_{\nu}(\sqrt{2}xu)}{K_{\nu}
 (\sqrt{2}au)}
p^{\nu}_{b/2a}(\lambda_{b},u)\: du,  &   x \geq a,
\end{array}
\right.
\end{equation*}
where $\lambda_{b}= \sqrt{(\lambda^2 + b^2)/2}$, $\lambda\in
\mathbb{R}$.
\end{theorem}

\begin{proof} The first statement results from
a combination of Corollary \ref{thm-straight-lines} and relation
(\ref{eq:density-bessel}). Next, using Lemmae \ref{doob} and
\ref{abs-diff}, with
 $\beta=-b/a$, we can write
\begin{eqnarray*}
& &\mathbb{E}_x^{\nu}\left[ e^{- \frac{\lambda^2}{2} T^{a-b\cdot}};
T^{a-b\cdot}<\frac{a}{b} \right]
\\&=& \mathbb{E}_x^{\nu}\left[
 e^{- \frac{\lambda^2}{2} \frac{H^{(b/a, a)} }{1+\frac{b}{a} H^{(b/a,a)}}}; H^{(b/a,a)}<\infty\right] \\
& = & \mathbb{E}_x^{\nu}\left[
 e^{- \frac{\lambda^2}{2} \frac{T^{a}}{1+\frac{b}{a} T^{a}}} \left(1+\frac{b}{a}
 T^{a}\right)^{-\delta/2}e^{ \frac{b}{2a}
\left(\frac{a^2}{1+\frac{b}{a} T^{a}}-x^2\right)}\right] \\
& = & e^{\frac{b}{2a}(a^2-x^2)} \mathbb{E}_x^{\nu}\left[
 e^{-\frac{\lambda_{b}^2 T^{a}}{1+\frac{b}{a} T^{a}}} \left(1+\frac{b}{a} T^{a}\right)^{-\delta/2}\right] \\
& = & e^{\frac{b}{2a}(a^2-x^2)} \int_0^{\infty}
p^{\nu}_{b/2a}(\lambda_{b},u) \mathbb{E}_x^{\nu}\left[ e^{- u^2
T^{a}} \right]du
\end{eqnarray*}
where the last line follows from \eqref{eq:sg}. It remains to
use (\ref{eq:laplace-bessel}) to conclude.
\end{proof}

\begin{remark} The process $(Y_t:=X_t+bt,t\geq 0)$ is a non-homogeneous Markov process and solves the sde
\begin{equation*}
Y_t=B_t+\frac{\delta-1}{2}\int_0^t\frac{ds}{Y_s-bs}+bt,\quad t\geq
0.
\end{equation*}
This is to be distinguished from a  Bessel process with a "naive"
drift $b$ introduced in \cite{Yor-84} and defined as a solution to
\begin{equation*}
Z_t=B_t+\frac{\delta-1}{2}\int_0^t\frac{ds}{Z_s}+bt,\quad t\geq 0.
\end{equation*}
\end{remark}
 \begin{remark} The Bessel process of dimension  $\delta=1$ is, in fact, the reflected Brownian motion.
It follows that the associated first hitting times  can be
interpreted as double barrier hitting times. That is, with $f$ as
above,  the time when a Brownian motion $B$ hits one of the curves
$x\pm f(\cdot)$, i.e. $\inf\{s>0; \: B_s=x\pm f(s)\}$.
\end{remark}

\section{Concluding remarks and some comments}
\subsection{}
It is plain that the results of Theorem \ref{thm} can be readily
extended to any $h$-transform of the process $X$, but we need to
take care of the life-times of the involved processes in the Bessel case. In the Brownian
setting, one gets similar results for the process
$X^{\epsilon}_t=X_t+y \epsilon t,\: t\geq0$, where $\epsilon$ is
an independent symmetric Bernoulli random variable taking  values
in $\{-1,1\}$. Observe that $i(X^{\epsilon})$ is a strong Markov
process. However, because the latter starts at the random point $
y \epsilon$, $X^{(\epsilon)}$ does not satisfy the time inversion
property \eqref{eq:ti}.

\subsection{}
 Assuming that $X$ is a $2$-self-similar strong Markov process with
c\`adl\`ag paths (i.e.~with possible jumps) enjoying the time
inversion property, then Theorem 5.8 in \cite{Pitman-Yor-81} or
Theorem 1 in \cite{Gallardo-Yor-05} allow to extend  Lemma
\ref{doob} and Lemma \ref{abs-diff} and an
analogue of Theorem \ref{thm}  can be
stated.  However, we did not succeed to construct examples of
$\mathbb{R}_+$-valued processes enjoying the time inversion
property other than Bessel processes in the usual or wide sense.
This is the reason why the setting is restricted to the continuous
one. We refer to \cite{Patie-08a} where the second author
characterizes, through its Mellin transform, the law of the first
passage time above the square root boundary for spectrally
negative positive $2$-self-similar Markov processes.

\subsection{}

 In \cite{Durbin-85}, Durbin  considered the studied boundary crossing problem
for continuous gaussian processes and showed that, in the
absolute-continuous case, the problem reduces to the computation
of a conditional expectation. In particular, for the Brownian motion, if
$f$ is continuously differentiable  and $f(0)\neq 0$ then \[ \P
\left(T^{f}\in dt \right)=\frac{1}{\sqrt{2\pi
t}}e^{-\frac{f^2(t)}{2t}}h(t) dt\] where
\begin{equation*}
h(t)=\lim_{s\nearrow t}\frac{1}{t-s}\mathbb{E}\left[B_s-f(s);
T^{f}>s \big{|}B_t=f(t) \right],\quad t>0.
\end{equation*}
 There seems to be no known way to compute the function
$h$. It is shown in \cite{Durbin-88} that this method is in
agreement with the standard method of images. With the obvious notation, it is clearly possible to relate $h^{(\beta)}$ to $h$.

\subsection{}
 We learned from Kendall \cite{Kendall-04} an intuitive
interpretation of Durbin's expression involving the family of local
times of $X$ at $f$ denoted by $ (L_{t}^{X=f}, t\geq 0)$. Indeed,
observe that Durbin's formula can be rewritten as $\PROB
\left(T^{f}\in dt \right)=h(t){\mathbb E}[dL_t^{X=f} ]$, $t\geq 0$.
The latter when integrated over $[0,a]$, yields ${\mathbb
P}\left(T^{f}<a \right)={\mathbb E}\left[\int_0^{a}h(s)dL_s^{X=f}
\right]$. Such a decomposition is not unique and many can be
constructed from the above one. A natural non trivial one to
consider is motivated by the following observation found in \cite{Kendall-04}. If
$g:\mathbb{R}^+\times \mathbb{R}^+\rightarrow \mathbb{R}^+$ solves
 ${\mathbb E}\left[\int_t^{a}g(s,a)dL_s^{X=f}\big|X_t=f(t)
\right]=1$, for $0< t<a$,
 then clearly ${\mathbb
E}\left[\int_0^{a}g(s,a)dL_s^{X=f} \right]={\mathbb P}\left(T^{f}<a
\right)$. However, it is not clear how to express $g$ in
 terms of $h$. W. Kendall  told the first author that the above interpretation could be interesting for further investigations for the studied boundary crossing problem.

\subsection{} In the Brownian case, it is interesting to analyze the impact of our identities on
some integral equations satisfied by the studied densities.
However, some of them lead to
 obvious facts explainable by change of variables. As an example, we observe that if $f$ is positive and does not vanish then $X$, when
started at $x>f(0)$, must hit $f$ before reaching $0$. The strong
Markov property then gives
\begin{equation*}
{x\over{\sqrt{2\pi t^3}}}e^{-{x^2\over {2t}}}=\int_0^t
\mathbb{P}_x(T^{f} \in dr) {f(r)\over {\sqrt{2\pi
{(t-r)^3}}}}e^{-{f(r)^2\over {2(t-r)}}},\:\:\:\: t>0.
\end{equation*}
This is easily shown to be in accordance  with the result stated in
Theorem \ref{thm}. For the above and other classical integral
equations we refer to \cite{Durbin-85},
 \cite{Ferebee-83}, \cite{Lerche-86} and also to
 \cite{Peskir-02-2} for some more recent ones.

\subsection{}

 The technics used in the example treated in  Subsection \ref{examples}, for Brownian motion,
can be applied
  to Bessel processes and
 square root curves. The required results, for that end, can be found in
Delong \cite{Delong-81} or in Yor \cite{Yor-84}.  This law is connected via a deterministic
time change to the one of the first passage time to a fixed level by the radial norm of a $\delta$-dimensional Ornstein-Uhlenbeck process, where $\delta$ is some positive integer. Due to the
 stationarity property,  the  law of the first passage times   can be
expressed as  infinite convolutions
 of mixture of exponential distributions, see Kent \cite{Kent-80}.

{\bf Acknowledgment:} We are grateful to  F.~Delbaen,
W.S.~Kendall,  D.~Talay and M.~Yor  for inspiring and fruitful
discussions on the topic. The first author would like to thank the
University of Bern for  their kind invitation for a visit during which
a part of this work was carried out. We thank two anonymous referees for their careful reading and useful comments that   helped to improve the presentation of  the paper.





%
%

\end{document}